\documentclass[11pt]{amsart}
\usepackage[T1]{fontenc}
\usepackage[utf8]{inputenc}
\usepackage{lmodern}
\usepackage{amsmath,amssymb,amsthm,mathtools}
\usepackage{xcolor}
\usepackage[colorlinks=true,linkcolor=blue!55!black,
  citecolor=blue!55!black,urlcolor=blue!55!black]{hyperref}

\newtheorem{theorem}{Theorem}[section]
\newtheorem{lemma}[theorem]{Lemma}
\newtheorem{corollary}[theorem]{Corollary}
\theoremstyle{definition}
\newtheorem{definition}[theorem]{Definition}
\newtheorem{remark}[theorem]{Remark}

\newcommand{\CSF}{\mathbf X}
\newcommand{\Core}{\mathcal C}
\newcommand{\wmult}{\mathsf m}
\title[Multiplicity-isolated cores]{Multiplicity-Isolated Cores and
Chromatic Symmetric Reconstruction of Trees}
\author{Zijian Zeng}
\address{Institute of Computer Science and Digital Innovation,
UCSI University, Kuala Lumpur 56000, Malaysia}
\email{2119516028@qq.com}
\email{1002266693@ucsiuniversity.edu.my}
\date{August 18, 2026}
\subjclass[2020]{05C05, 05E05, 05C60}
\keywords{chromatic symmetric function, tree reconstruction, star basis,
proper tree, diameter six}
\hypersetup{pdftitle={Multiplicity-Isolated Cores and Chromatic Symmetric Reconstruction of Trees},
pdfauthor={Zijian Zeng}}

\begin{document}

\begin{abstract}
Stanley's tree-isomorphism conjecture asks whether the chromatic symmetric
function distinguishes nonisomorphic trees.  We give a reconstruction
criterion that permits repeated leaf-component orders.  For a proper tree,
collapse each leaf component to its center and record its order as a vertex
weight.  We prove that the chromatic symmetric function reconstructs the
tree whenever every nonleaf vertex of this weighted core has a weight that
occurs nowhere else in the core.  Repetitions among core leaves are
unrestricted.  The proof uses only the leading star-basis partition and the
coefficients immediately above it.  As a consequence, the conjecture holds
for an infinite class of diameter-six trees not covered by the condition
that all leaf-component orders are distinct.  We also give a canonical
integer-partition model for arbitrary diameter-six trees and an exact
cut-partition implementation intended for further work.  The unrestricted
diameter-six case remains open.
\end{abstract}

\maketitle

\section{Introduction}

For a finite simple graph $G$, Stanley's chromatic symmetric function is
\[
  \CSF_G=\sum_{\kappa}\prod_{v\in V(G)}x_{\kappa(v)},
\]
where the sum is over all proper colorings by positive integers
\cite{Stanley_1995}.  Stanley asked whether $\CSF_T$ distinguishes trees up
to isomorphism.  The conjecture remains open in general.

Small-diameter reconstruction has recently advanced through the star basis
and the deletion-near-contraction relation.  Aliste-Prieto, de Mier,
Orellana, and Zamora proved the conjecture for proper trees of diameter at
most five \cite{Aliste_Prieto_2023}.  Gonzalez, Orellana, and Tomba then
removed the properness assumption and reconstructed every tree of diameter
strictly less than six \cite{GonzalezOrellanaTomba2025}.  Diameter six is the
next case not covered by that theorem.

The main result of this note is a sufficient condition expressed on the
weighted core of a proper tree.  Its point is that terminal leaf-component
orders may repeat with arbitrary multiplicity.  Thus the result is strictly
broader than the immediate reconstruction criterion in which every part of
the leading partition is distinct.

\begin{theorem}\label{thm:intro}
Let $T$ be a proper tree.  Suppose that, among the leaf components of $T$,
the order of every component whose center is a nonleaf of the internal core
occurs exactly once.  Then $T$ is determined up to isomorphism by
$\CSF_T$.
\end{theorem}

The proof is short but useful: the leading star-basis partition gives the
multiset of component orders, and the next layer gives the number of core
edges of every endpoint-order type.  Under the hypothesis, repeated orders
belong only to core leaves, so these two pieces of data determine the whole
weighted core.

Theorem~\ref{thm:intro} applies in every diameter.  In diameter six it gives
an infinite class.  For example, take a five-vertex path as weighted core,
give its two endpoints a common weight $r\ge2$, and give its three internal
vertices pairwise distinct weights, all different from $r$.  Expanding a
weight $w$ into a center with $w-1$ pendant leaves produces a proper tree of
diameter six covered by the theorem, although its leading partition has a
repeated part.

\section{Leaf components and weighted cores}

An edge of a tree is a \emph{leaf edge} if one endpoint is a leaf.  The
\emph{internal core} $\Core(T)$ is the tree induced by the nonleaf vertices
of $T$.  A tree is \emph{proper} if every vertex of $\Core(T)$ is adjacent
in $T$ to at least one leaf.  For such a tree define
\[
 \omega_T(v)=1+\#\{u\in V(T):u\text{ is a leaf and }uv\in E(T)\},
 \qquad v\in V(\Core(T)).
\]
Every weight is at least two.  Conversely, $T$ is recovered from the
strictly weighted tree $(\Core(T),\omega_T)$ by attaching
$\omega_T(v)-1$ leaves to each core vertex $v$.

The connected components obtained after deleting all internal edges are
stars, called the \emph{leaf components}.  The component centered at $v$
has order $\omega_T(v)$.  Thus the leaf-component partition is
\[
  \lambda_{\rm LC}(T)
  =\{\!\{\omega_T(v):v\in V(\Core(T))\}\!\}.
\]

\begin{definition}
The weighted core of a proper tree is \emph{multiplicity-isolated} if
\begin{equation}\label{eq:MI}
 \deg_{\Core(T)}(v)\ge2
 \quad\Longrightarrow\quad
 \omega_T^{-1}(\omega_T(v))=\{v\}.
\end{equation}
Equivalently, every weight that occurs more than once occurs exclusively on
leaves of $\Core(T)$.
\end{definition}

We use the following information recovered from the star-basis expansion.

\begin{lemma}[Star-basis edge data]\label{lem:edge-data}
Let $T$ be a proper tree.  The chromatic symmetric function $\CSF_T$
determines
\begin{enumerate}
\item the multiset
  $\mathcal W_T=\{\!\{\omega_T(v):v\in V(\Core(T))\}\!\}$; and
\item for every unordered pair $\{p,q\}$, the integer
\[
 e_{pq}(T)=\#\{uv\in E(\Core(T)):
       \{\omega_T(u),\omega_T(v)\}=\{p,q\}\}.
\]
\end{enumerate}
\end{lemma}

\begin{proof}
The leading partition of the star-basis expansion is the leaf-component
partition, giving (1).  Since $T$ is proper, this partition has no part of
size one.  Proposition~4.13(a) of
\cite{GonzalezOrellanaTomba2025} shows that the coefficients indexed by
partitions of length one less than the leading partition recover exactly the
numbers of internal edges with each unordered pair of leaf-component orders
as endpoints.  This gives (2).  An equivalent strictly weighted
$D$-polynomial formulation appears in Lemma~7.12 of
\cite{Aliste_Prieto_2023}.
\end{proof}

\section{Reconstruction from multiplicity-isolated data}

The graph-theoretic step is independent of symmetric functions.

\begin{lemma}\label{lem:weighted-reconstruction}
Let $(C,\omega)$ be a weighted tree with at least three vertices.  Assume
that every nonleaf vertex has a weight of multiplicity one.  Then $(C,\omega)$
is determined up to weight-preserving isomorphism by
\begin{enumerate}
\item the multiplicities $\wmult_p=|\omega^{-1}(p)|$; and
\item the numbers
 $e_{pq}=\#\{uv\in E(C):\{\omega(u),\omega(v)\}=\{p,q\}\}$.
\end{enumerate}
\end{lemma}

\begin{proof}
For every $p$ with $\wmult_p=1$, create a single vertex $v_p$ of weight
$p$.  For every $p$ with $\wmult_p>1$, create $\wmult_p$ indistinguishable
vertices of weight $p$.  By hypothesis, all vertices in the latter class
are leaves.

If $\wmult_p=\wmult_q=1$, join $v_p$ and $v_q$ precisely when
$e_{pq}=1$.  If $\wmult_p>1$ and $\wmult_q=1$, attach exactly $e_{pq}$ of
the weight-$p$ leaves to $v_q$.  There is no edge between two repeated
weight classes: such an edge would join two leaves, which is impossible in
a tree on at least three vertices.  These rules account for every edge and
every repeated-weight vertex, because
\[
  \sum_{q:\,\wmult_q=1} e_{pq}=\wmult_p
  \qquad(\wmult_p>1).
\]
The construction is forced, and choices among vertices of the same repeated
weight differ only by a weight-preserving isomorphism.
\end{proof}

\begin{proof}[Proof of Theorem~\ref{thm:intro}]
By Lemma~\ref{lem:edge-data}, $\CSF_T$ determines the input data of
Lemma~\ref{lem:weighted-reconstruction}.  Condition
\eqref{eq:MI} is exactly its multiplicity hypothesis.  Hence $\CSF_T$
determines the weighted core.  Attaching $\omega_T(v)-1$ leaves at every
recovered core vertex determines $T$.

If the core has one vertex, $T$ is a star.  If it has two vertices, $T$ is a
bi-star; the endpoint weights determine it, including the equal-weight
case.  Thus the small cores are covered as well.
\end{proof}

\begin{corollary}\label{cor:diameter-six}
Stanley's tree-isomorphism conjecture holds for every proper tree of
diameter six whose weighted core is multiplicity-isolated.
\end{corollary}

\begin{remark}
Repeated terminal types are genuinely unrestricted.  A repeated weight
may occur on arbitrarily many leaves of the core and may be distributed
among several uniquely weighted nonleaf vertices; the numbers $e_{pq}$
recover that distribution.
\end{remark}

\section{A canonical model for diameter six}

For completeness, we record a compact parametrization useful beyond
Corollary~\ref{cor:diameter-six}.  A diameter-six tree has a unique central
vertex $c$.  Every component of $T-c$, rooted at its neighbor of $c$, has
height at most two.  Such a rooted branch is encoded by an integer partition
$\pi$: a part $1$ denotes a leaf child of the branch root, while a part
$q\ge2$ denotes a child supporting $q-1$ leaves.  The branch order is
$1+|\pi|$, and the empty partition encodes a one-vertex branch.

Consequently, a diameter-six tree is canonically represented by a multiset
$\{\!\{\pi_1,\ldots,\pi_s\}\!\}$ satisfying
\[
  |V(T)|=1+\sum_{i=1}^s(1+|\pi_i|),
\]
with at least two $\pi_i$ containing a part at least two.  This is a
bijection with unlabelled diameter-six trees.

Stanley's power-sum formula
\[
 \CSF_T=\sum_{A\subseteq E(T)}(-1)^{|A|}p_{\lambda(A)}
\]
also gives an exact computational invariant: for a tree, the sign is
determined by the length of $\lambda(A)$, so $\CSF_T$ is equivalent to the
table counting edge subsets by their component-order partition.  The
supplementary implementation computes this table by integer dynamic
programming and independently checks small orders by direct edge-subset
enumeration.  These computations are diagnostic only and are not used in
the proof.

\section{Limitations and further questions}

The unrestricted diameter-six problem is not resolved here.  Two obstacles
remain outside the theorem:
\begin{enumerate}
\item a nonproper tree produces weight-one core vertices, for which the
  simplest adjacency extraction changes form; and
\item if a nonleaf core weight is repeated, the aggregate numbers $e_{pq}$
  need not separate the individual vertices in that weight class.
\end{enumerate}
Higher star-basis layers record contractions of several internal edges and
are natural candidates for resolving the second obstruction.  The canonical
partition model above turns the full diameter-six question into recovery of
a multiset of height-two rooted branches from exact cut-partition data.

\section*{Data and code availability}

The accompanying source archive contains the exact cut-partition evaluator,
the canonical diameter-six generator, and a separate finite checker for the
weighted-core reconstruction lemma.  No computational output is used as a
premise of Theorem~\ref{thm:intro}.

\bibliographystyle{amsplain}
\bibliography{references}

\end{document}